\documentclass[twoside,leqno,10pt]{amsart}
\usepackage{amsfonts}
\usepackage{amsmath}
\usepackage{amscd}
\usepackage{amssymb}
\usepackage{amsthm}
\usepackage{amsrefs}
\usepackage{latexsym}
\usepackage{bbm}
\begin{document}

\newtheorem{theorem}[subsection]{Theorem}
\newtheorem{proposition}[subsection]{Proposition}
\newtheorem{lemma}[subsection]{Lemma}
\newtheorem{corollary}[subsection]{Corollary}
\newtheorem{conjecture}[subsection]{Conjecture}
\newtheorem{prop}[subsection]{Proposition}
\numberwithin{equation}{section}
\newcommand{\mr}{\ensuremath{\mathbb R}}
\newcommand{\dif}{\mathrm{d}}
\newcommand{\intz}{\mathbb{Z}}
\newcommand{\ratq}{\mathbb{Q}}
\newcommand{\natn}{\mathbb{N}}
\newcommand{\comc}{\mathbb{C}}
\newcommand{\rear}{\mathbb{R}}
\newcommand{\prip}{\mathbb{P}}
\newcommand{\uph}{\mathbb{H}}
\newcommand{\fief}{\mathbb{F}}
\newcommand{\majorarc}{\mathfrak{M}}
\newcommand{\minorarc}{\mathfrak{m}}
\newcommand{\sings}{\mathfrak{S}}
\newcommand{\fA}{\ensuremath{\mathfrak A}}
\newcommand{\mn}{\ensuremath{\mathbb N}}
\newcommand{\mq}{\ensuremath{\mathbb Q}}
\newcommand{\mc}{\ensuremath{\mathbb C}}
\newcommand{\half}{\tfrac{1}{2}}
\newcommand{\f}{f\times \chi}
\newcommand{\summ}{\mathop{{\sum}^{\star}}}
\newcommand{\chiq}{\chi \bmod q}
\newcommand{\chidb}{\chi \bmod db}
\newcommand{\chid}{\chi \bmod d}
\newcommand{\sym}{\text{sym}^2}
\newcommand{\hhalf}{\tfrac{1}{2}}
\newcommand{\sumstar}{\sideset{}{^*}\sum}
\newcommand{\sumprime}{\sideset{}{'}\sum}
\newcommand{\sumprimeprime}{\sideset{}{''}\sum}
\newcommand{\shortmod}{\ensuremath{\negthickspace \negthickspace \negthickspace \pmod}}
\newcommand{\V}{V\left(\frac{nm}{q^2}\right)}
\newcommand{\sumi}{\mathop{{\sum}^{\dagger}}}
\newcommand{\mz}{\ensuremath{\mathbb Z}}
\newcommand{\leg}[2]{\left(\frac{#1}{#2}\right)}
\newcommand{\muK}{\mu_{\omega}}

\title[Moments and one level density of unitary families of {H}ecke {$L$}-functions]{Moments and One level density of certain unitary families of {H}ecke {$L$}-functions}
\date{\today}
\author{Peng Gao and Liangyi Zhao}

\begin{abstract}
In this paper, we study moments of central values of certain unitary families of Hecke $L$-functions of the Gaussian field, and establish quantitative non-vanishing result for the central values.  We also establish a one level density result for the low-lying zeros of these families of Hecke $L$-functions.
\end{abstract}

\maketitle


\section{Introduction}
\label{sec1}

The non-vanishing of central values of $L$-functions is of central importance in number theory.  In the classical case of Dirichlet $L$-functions, S. Chowla \cite{chow} conjectured that $L(1/2, \chi) \neq 0$ for every primitive Dirichlet character $\chi$.  One typical way to investigate this non-vanishing problem is to study the moments of a family of $L$-functions. By considering the first and second mollified moments of $L(1/2, \chi)$, B. Balasubramanian and V. K. Murty \cite{BM} showed that $L(1/2,\chi) \neq 0$ for at least $4\%$ of Dirichlet characters $\chi \mod{q}$. For primitive characters, the proportion was improved to $1/3$ by H. Iwaniec and P. Sarnak in \cite{I&S},  to $34.11\%$ by H. M. Bui [Bu] and most recently to $3/8$ by R. Khan and H. T. Ngo \cite{KN}. \newline

Instead of mollified moments, one may be only interested in the moments of Dirichlet $L$-functions. The first moment of the family of primitive Dirichlet $L$-functions of modulus $q$ has long been known while the second moment is due to R. E. A. C. Paley \cite{Paley}. In \cite{HB2}, D. R. Heath-Brown obtained an asymptotic formula for the fourth moment of the family of $L$-functions associated with primitive Dirichlet characters modulo $q$, provided $q$ does not have too many distinct prime divisors. The formula was extended to all integers by K. Soundararajan in \cite{sound2}. An asymptotic formula for prime moduli with power savings was obtained by M. P. Young in \cite{Young1} and this result was later improved by V. Blomer, E. Fouvry, E. Kowalski, P. Michel and D. Mili\'cevi\'c \cite{BFKMM}. \newline

As an analogue of Dirichlet $L$-functions, T. Stefanicki \cite{Stefanicki} obtained the first and second moments of Dirichlet twists of modular $L$-functions. The formula for the second moment is valid for a density zero set and is extended to almost all integers in \cite{GKR}. \newline

  Motivated by the result of Stefanicki, we consider in this paper a family of Hecke $L$-functions in the Gaussian field. Throughout this paper, we let $K=\mq(i)$ and $\mathcal{O}_K=\mz[i]$ for the ring of integers in $K$. We also denote $U_K=<i>$ for the group of units in $\mathcal{O}_K$. Let $q \in \mathcal{O}_K$ with $(q, 2)=1$ and $\chi$ be a homomorphism:
\begin{align}
\label{chi}
  \chi: \left (\mathcal{O}_K / (q) \right )^*  \rightarrow S^1 :=\{ z \in \mc | \hspace{0.1in} |z|=1 \}.
\end{align}
  We shall say $\chi$ is a character modulo $q$.  Note that in $\mathcal{O}_K$, every ideal co-prime to $2$ has a unique
generator congruent to $1$ modulo $(1+i)^3$ (see the paragraph above Lemma 8.2.1 in \cite{BEW}).  Such a generator is
called primary.  When $q$ is co-prime to $2$, $\chi$ induces a character $\widetilde{\chi}$ modulo $(1+i)^3q$.  To see this, note that the ring
$(\mathcal{O}_K/(1+i)^3q)^*$ is isomorphic to the direct product of the group of units $U_K$ and the group $N_{q}$ formed by elements in $(\mathcal{O}_K/(1+i)^3q)^*$ and congruent to $1 \pmod {(1+i)^3}$ (i.e., primary). Under this isomorphism, any element $n \in (\mathcal{O}_K/(1+i)^3q)^*$ can be written uniquely as $n=u_n \cdot n_0$ with $u_n \in U_K$, $n_0 \in N_{q}$.  We can now define $\widetilde{\chi} \pmod {(1+i)^3q}$ such that for any $n \in (\mathcal{O}_K/(1+i)^3q)^*$,
\begin{align*}
   \widetilde{\chi}(n)=\chi(n_0).
\end{align*}
   We say that $\chi$ is a primitive character modulo $q$ if it does not factor through $\left (\mathcal{O}_K / (q') \right )^*$ for any proper divisor $q'$ of $q$. When $\chi$ is primitive and $\chi(-1)=-1$, we will show in Section \ref{sec2.4} that the character $\widetilde{\chi}$ is also primitive modulo $(1+i)^3q$.  As $\widetilde{\chi}$ is primitive and trivial on units, it follows from the discussions on  \cite[p. 59-60]{iwakow} that $\widetilde{\chi}$ can be regarded as a primitive Hecke character $\pmod {(1+i)^3q}$ of trivial infinite type. We denote $\widetilde{\chi}$ for this Hecke character as well. In the rest of the paper, unless otherwise specified, we shall always regard $\widetilde{\chi}$ as a Hecke character. \newline

   Let $\psi^*(q)$ denote the number of primitive characters $\chi \pmod q$ satisfying $\chi(-1)=-1$ and let $\omega(q)$ denote the number of distinct prime ideals dividing $(q)$. Our first result is the following
\begin{theorem}
\label{moments}
   For $q \in \mathcal{O}_K, (q,2)=1$ and any $\varepsilon > 0$, we have, as $N(q) \rightarrow \infty$,
\begin{equation} \label{firstmoment}
\sumstar_{\substack{\chiq \\ \chi(-1)=-1}} L \left( \frac{1}{2} , \widetilde{\chi} \right)
=\frac 12 \psi^*(q)+O(2^{\omega(q)}N(q)^{1/2+\varepsilon}),
\end{equation}
and
\begin{align} \label{secondmoment}
 \sumstar_{\substack{\chiq \\ \chi(-1)=-1}} \left| L \left( \frac{1}{2}, \widetilde{\chi} \right) \right|^2
=& \left ( \frac {\pi}{16} \frac{\varphi(q)}{N(q)}  \log N(q)+\frac {\pi}{8} \frac{\varphi(q)}{N(q)}\sum_{\substack{\mathfrak{p} |2q}}  \frac {\log N(\mathfrak{p})}{N(\mathfrak{p})-1}  +\frac{\varphi(q)}{N(q)}C_0 \right )\psi^{*}(q) \\
& \hspace*{2cm} +O\Big(N(q)^{3/4+\varepsilon} \Big ).  \nonumber
\end{align}
   Here the $*$ on the sum over $\chi$ restricts the sum to primitive characters, $C_0>0$ is an explicitly computable constant and $\varphi(q) = \# (\mathcal{O}_K / (q))^*$.
\end{theorem}

   We note here the asymptotic formulas in Theorem \ref{moments} are valid for all large $N(q)$ because of the lower bound for $\psi^{*}(q)$ given in \eqref{psibound} and if $N(q) \geq 3$, then (see \cite[(2.1)]{Schaal2})
\begin{align}
\label{omega}
  \omega(q) \ll \frac {\log N(q)}{\log \log N(q)}.
\end{align}

  We readily deduce from Theorem \ref{moments} via a standard argument using Cauchy's inequality (see \cite[p. 568]{BM}), the following
\begin{corollary}
\label{cor1}
  For $q \in \mathcal{O}_K, (q,2)=1$, we have as $N(q) \rightarrow \infty$,
\begin{equation*}
   \# \left\{ \widetilde{\chi}: \chi \bmod {q}, \chi(-1)=-1, \chi \hspace{0.05in} \mbox{primitive}, L\left( \frac{1}{2}, \widetilde{\chi} \right) \neq 0 \right\} \gg
  \frac {\psi^*(q)}{\log N(q)}.
\end{equation*}
\end{corollary}

  Note that Corollary \ref{cor1} does not establish that $L\left( \frac{1}{2}, \widetilde{\chi} \right) \neq 0$ for a positive proportion of the characters $\chi$ to a given modulus. To obtain a positive proportion result, other than studying the mollified moments, we can also study the $1$-level densities of low-lying zeros of families of $L$-functions.  The density conjecture of N. Katz and P. Sarnak \cites{KS1, K&S} suggests that the distribution of zeros near $1/2$ of a family of $L$-functions is the same as that of eigenvalues near $1$ of a corresponding classical compact group. This conjecture implies that $L(1/2, \chi) \neq 0$ for almost all primitive Dirichlet $L$-functions.  Assuming the generalized Riemann hypothesis (GRH), M. R. Murty  \cite{Murty} showed that at least $50\%$ of both primitive Dirichlet $L$-functions and Dirichlet twists of modular $L$-functions do not vanish at the central point. The result of Murty can be regarded as the $1$-level density of low-lying zeros of the corresponding families of $L$-functions for test functions whose Fourier transforms being supported in $[-2, 2]$. In \cite{HuRu}, H. P. Hughes and Z. Rudnick studied the $1$-level density of low-lying zeros of the family of primitive Dirichlet $L$-functions of a fixed prime modulus. Their work shows that this family is a unitary family. \newline

   Our next result concerns the $1$-level density of low-lying zeros of the family $\{ L(s, \widetilde{\chi}) \}$ of Hecke $L$-functions in $\mq(i)$. Here $\chi$ runs over primitive characters modulo $q$ satisfying $\chi(-1)=-1$ with $q \in \mz[i],(q,2)=1$.
We denote the non-trivial zeroes of the Hecke $L$-function
   $L(s, \widetilde{\chi})$ by $1/2+i \gamma_{\widetilde{\chi}, j}$.  Without assuming GRH, we order them as
\begin{equation*}
    \ldots \leq
   \Re \gamma_{\widetilde{\chi}, -2} \leq
   \Re \gamma_{\widetilde{\chi}, -1} < 0 \leq \Re \gamma_{\widetilde{\chi}, 1} \leq \Re \gamma_{\widetilde{\chi}, 2} \leq
   \ldots.
\end{equation*}
    We set
\begin{align*}
    \tilde{\gamma}_{\widetilde{\chi}, j}= \frac{\gamma_{\widetilde{\chi}, j}}{2 \pi} \log N(q)
\end{align*}
and define for an even Schwartz class function $\phi$,
\begin{equation*}
S(\widetilde{\chi}, \phi)=\sum_{j} \phi(\tilde{\gamma}_{\widetilde{\chi}, j}).
\end{equation*}

  Following \cite[Definition 2.1]{HuRu}, we say a function $f(x)$ is an admissible function if it is a real, even function, whose
Fourier transform $\hat{f}(u)$ is compactly supported, and such that $f(x) \ll (1+|x|)^{-1-\delta}$ for some $\delta>0$.
  Our result is
\begin{theorem}
\label{onelevelunitary}
 Let $\phi(x)$ be an admissible function whose
Fourier transform $\hat{\phi}(u)$ has compact support in $(-2, 2)$. Then for $q \in \mz[i], (q,2)=1$, we have
\begin{align}
\label{unitarydensity}
 \lim_{N(q) \rightarrow \infty}\frac{1}{\psi^*(q)}\sumstar_{\substack{\chiq \\ \chi(-1)=-1}}  S(\widetilde{\chi}, \phi)
 = \int\limits_{\mathbb{R}} \phi(x) \dif x.
\end{align}
   Here the $*$ on the sum over $\chi$ restricts the sum to primitive characters.
\end{theorem}

    Theorem \ref{onelevelunitary} can be regarded as an analogue to the above mentioned result of H.P. Hughes and Z. Rudnick in \cite{HuRu}. The left-hand side expression of \eqref{unitarydensity} is known as the $1$-level density of low-lying zeros of the family $\{ L(s, \widetilde{\chi}) \}$.   In connection with the random matrix theory (see the discussions in \cite{G&Zhao2}), the right-hand side expression of \eqref{unitarydensity} shows that the family is also a unitary family.   \newline

   Using the argument in the proof of \cite[Corollary 1.4]{G&Zhao4}, we deduce readily a  positive proportion non-vanishing result for the family of Hecke $L$-functions under our consideration.
 \begin{corollary}
 Suppose that the GRH is true and that $1/2$ is a zero of $L \left( s, \widetilde{\chi} \right)$ of order $n_{\widetilde{\chi}} \geq 0$.  As $N(q) \to \infty$,
 \[  \sumstar_{\substack{\chiq \\ \chi(-1)=-1}} n_{\widetilde{\chi}} \leq  \left( \frac{1}{2} + o(1) \right) \psi^*(q). \]
   Moreover, as $N(q) \to \infty$
\[   \# \left\{ \widetilde{\chi}: \chi \bmod {q}, \chi(-1)=-1, \chi \hspace{0.05in} \mbox{primitive}, L\left( \frac{1}{2}, \widetilde{\chi} \right) \neq 0 \right\} \geq \left( \frac{1}{2} + o(1) \right) \psi^*(q) . \]
    \end{corollary}

\subsection{Notations} The following notations and conventions are used throughout the paper.\\
\noindent $e(z) = \exp (2 \pi i z) = e^{2 \pi i z}$. \newline
$f =O(g)$ or $f \ll g$ means $|f| \leq cg$ for some unspecified
positive constant $c$. \newline
$f =o(g)$ means $\displaystyle \lim_{x \rightarrow \infty}f(x)/g(x)=0$. \newline
$K=\mq(i), \mathcal{O}_K=\mz[i]$. \newline
$\mu_{[i]}$ denotes the M\"obius function on $\mathcal{O}_K$. \newline
$\varphi$ denotes Euler's totient function on $\mathcal{O}_K$. \newline
$\varpi$ denotes a prime in $K$.

\section{Preliminaries}
\label{sec 2}
\subsection{Orthogonality relations and primitive Hecke characters}
\label{sec2.4}
  Let $q \in \mathcal{O}_K, (q, 2)=1$ and let $\chi$ be a primitive character modulo $q$ defined in \eqref{chi} satisfying $\chi(-1)=-1$. We note the following orthogonality relations. As the proof is similar to the classical case (see \cite[Lemma 1]{sound2}), we omit it here.
\begin{lemma}
\label{ortho}
  Let $q \in \mathcal{O}_K, (q, 2)=1$.  Let $a=\pm 1$,  we have for $(nm,q)=1$
\begin{align*}
\sumstar_{\substack {\chiq \\ \chi(-1)=(-1)^a}} \chi(n)\overline{\chi}(m) = \frac 12 \sum_{\substack{d|q\\
n\equiv m \bmod d}} \mu_{[i]}(q/d) \varphi(d)+ \frac {(-1)^a}2 \sum_{\substack{d|q\\
n\equiv -m \bmod d}} \mu_{[i]}(q/d) \varphi(d).
\end{align*}
\end{lemma}

    By setting $n=m=1$ in Lemma \ref{ortho}, we deduce immediately the following
\begin{corollary}
\label{psi}
  Let $q \in \mathcal{O}_K, (q, 2)=1$ and let $\psi^*(q)$ denote the number of primitive characters $\chi \pmod q$ satisfying $\chi(-1)=-1$, then
\begin{align*}
  \psi^*(q)=\frac 12\psi(q)-\frac 12\mu_{[i]}(q),
\end{align*}
 where $\psi(q)$ denotes the number of primitive characters $\chi \pmod q$. Moreover, $\psi(q)$ is a multiplicative function given by $\psi(\varpi) = N(\varpi)- 2$ for primes $\varpi$, and $\psi(\varpi^k) = N(\varpi)^k(1 -1/N(\varpi))^2$ for $k \geq  2$.

\end{corollary}

   We note that Corollary \ref{psi} implies that for $(q,2)=1$, we have
\begin{align}
\label{psibound}
  \psi^*(q), \psi(q) \gg N(q) \left ( \frac {\varphi(q) }{N(q) }\right )^2 \gg \frac {N(q)}{\log \log N(q)}.
\end{align}

   Now we show that the induced character $\widetilde{\chi}$ modulo $(1+i)^3q$ is also primitive. Suppose that $\widetilde{\chi}$ is induced by a character modulo $(1+i)^3q'$ for some proper divisor $q'$ of $q$. Then as $\chi$ is primitive, there exists a $c \equiv 1 \pmod {q'}$ such that $\chi(c) \neq 1$. By the Chinese Remainder Theorem, we can then find a $c_0$ such that $c_0 \equiv 1 \pmod {(1+i)^3}$ and $c_0 \equiv c \pmod q$. It follows from our definition that $\widetilde{\chi}(c_0)=\chi(c) \neq 1$. This contradiction shows that $\widetilde{\chi}$ can only be possibly induced by a character $\chi'$ modulo $(1+i)^2q$. But in this case, we can again apply the Chinese Remainder Theorem to find a $c_0$ such that
$c_0 \equiv -1 \pmod {(1+i)^3}$ and $c_0 \equiv 1 \pmod q$.  As $-1 \equiv 1 \pmod {(1+i)^2}$, we have $c_0 \equiv 1 \pmod {(1+i)^2q}$ so that $\chi'(c_0)=1$. However, it follows from the definition that $\widetilde{\chi}(c_0)=\chi(-c_0)=-1$. This implies that $\widetilde{\chi}$ can not be induced by $\chi'$ either and hence is primitive.

\subsection{The approximate functional equation}
\label{AFE}

      Let $\widetilde{\chi}$ be given as in the previous section regarding as a primitive Hecke character modulo $(1+i)^3q$ of trivial infinite type. The Hecke $L$-function associated with this Hecke character $\widetilde{\chi}$ is defined for $\Re(s) > 1$ by
\begin{equation*}
  L(s, \widetilde{\chi} ) = \sum_{0 \neq \mathcal{A} \subset
  \mathcal{O}_K}\widetilde{\chi}(\mathcal{A})(N(\mathcal{A}))^{-s},
\end{equation*}
  where $\mathcal{A}$ runs over all non-zero integral ideals in $K$ and $N(\mathcal{A})$ is the
norm of $\mathcal{A}$. As shown by E. Hecke, $L(s, \widetilde{\chi})$ admits
analytic continuation to an entire function and satisfies a
functional equation (see \cite[Corollary 8.6]{Newkirch}):
\begin{equation}
\label{1.1}
  \Lambda(s, \widetilde{\chi}) = g(\widetilde{\chi})(N((1+i)^3q))^{-1/2}\Lambda(1-s, \overline{\widetilde{\chi}}),
\end{equation}
   where $D_K=-4$ is the discriminant of $K$, $g(\widetilde{\chi})$ is the Gauss sum defined by
\begin{equation*}
   g(\widetilde{\chi})=\sum_{x \bmod (1+i)^3q} \widetilde{\chi}(x) \widetilde{e}\leg{x}{(1+i)^3q}, \quad \widetilde{e}(z) =e \left( \text{tr} \left( {\frac {z}{2i}} \right) \right),
\end{equation*}
  and
\begin{equation*}
  \Lambda(s, \widetilde{\chi}) = (|D_K|N((1+i)^3q))^{s/2}(2\pi)^{-s}\Gamma(s)L(s, \widetilde{\chi}).
\end{equation*}
   We refer the reader to \cite{Newkirch} for a more detailed discussion of the Hecke characters and $L$-functions.  \newline

   Note that we have $|g(\widetilde{\chi})|=(N((1+i)^3q))^{1/2}$ (see \cite[Exercise 12, p. 61]{iwakow}) and that it follows from the definition that $\overline{g(\widetilde{\chi})}=\widetilde{\chi}(-1)g(\overline{\widetilde{\chi}})=g(\overline{\widetilde{\chi}})$, as $\widetilde{\chi}(-1)(-1)=1$.
From this and \eqref{1.1}, we get that
\begin{equation}
\label{1.1'}
  \Lambda \left( \frac 12+s, \widetilde{\chi} \right) \Lambda \left( \frac 12+s, \overline{\widetilde{\chi}} \right) =\Lambda \left( \frac 12-s, \widetilde{\chi} \right) \Lambda \left( \frac 12-s, \overline{\widetilde{\chi}} \right).
\end{equation}

    For $c>1/2$ we consider
$$
I:= \frac{1}{2\pi i} \int\limits_{(c)}
\frac{\Lambda(1/2+s, \widetilde{\chi}) \Lambda(1/2+s, \overline{\widetilde{\chi}}) }{\Gamma(1/2)^2}
\frac{ds}{s}.
$$
We move the line of integration to Re$(s)=-c$ and use the
relation \eqref{1.1'} to see that
$I=|L(1/2,\widetilde{\chi})|^2 -I$, so that $|L(1/2,\widetilde{\chi})|^2 =2I$.
On the other hand, expanding $L(1/2+s,\widetilde{\chi})L(1/2+s,\overline{\widetilde{\chi}})$ into
its Dirichlet series and integrating termwise, we get
$I=A(\widetilde{\chi})$, where
\begin{align}
\label{achi}
A(\widetilde{\chi}) := \sum_{0 \neq \mathcal{A}, \mathcal{B} \subset
  \mathcal{O}_K} \widetilde{\chi}(\mathcal{A})\overline{\widetilde{\chi}(\mathcal{B})}
(N(\mathcal{A})N(\mathcal{B}))^{-1/2} W\left( \frac {N(\mathcal{A})N(\mathcal{B})}{N(q)}\right),
\end{align}
 with
$$
W(x) =
\frac{1}{2\pi i} \int\limits_{(c)}
\left( \frac{\Gamma(s+1/2)}
{\Gamma(1/2)}\right)^2 \left( \frac {2|D_k|}{\pi^2} \right)^s x^{-s} \frac{ds}{s},
$$
for any positive $x, c$.  Similar to \cite[(1.3a), (1.3b)]{sound2}, we have for any $j \geq 0$,
\begin{align}
\label{w}
W(x) = 1+ O(x^{1/2 -\epsilon}), \quad W^{(j)}(x) = O_c (x^{-c} ).
\end{align}

 On the other hand, we note the following expression for $L(1/2, \widetilde{\chi})$ (see \cite[Section 2.3]{G&Zhao3}):
\begin{equation} \label{approxfuneq}
\begin{split}
 L \left( \frac{1}{2}, \widetilde{\chi} \right) = \sum_{0 \neq \mathcal{A} \subset
  \mathcal{O}_K}  & \frac{\widetilde{\chi}(\mathcal{A})}{N(\mathcal{A})^{1/2}}V \left(\frac{N(\mathcal{A})}{x} \right) \\
  & + \frac{g(\widetilde{\chi})}{N((1+i)^3q)^{1/2}}\sum_{0 \neq \mathcal{A} \subset
  \mathcal{O}_K}\frac{\overline{\widetilde{\chi}}(\mathcal{A})}{N(\mathcal{A})^{1/2}}V\left(\frac{
  N(\mathcal{A})x}{|D_K|N((1+i)^3q)} \right),
     \end{split}
\end{equation}
     where $x>0$ and
\begin{align*}
  V \left(\xi \right)=\frac {1}{2\pi
   i}\int\limits\limits_{(2)}\frac {\Gamma(s+1/2)}{\Gamma (1/2)}\frac
   {(2\pi\xi)^{-s}}{s} \ \dif s.
\end{align*}

   We note (see \cite[Lemma 2.1]{sound1}) the following estimation for the $j$-th derivative of $V(\xi)$:
\begin{equation} \label{2.07}
      V\left (\xi \right) = 1+O(\xi^{1/2-\varepsilon}) \; \mbox{for} \; 0<\xi<1   \quad \mbox{and} \quad V^{(j)}\left (\xi \right) =O(e^{-\xi}) \; \mbox{for}
      \; \xi >0, \; j \geq 0.
\end{equation}

\subsection{The explicit formula}
\label{section: Explicit Formula}

Our approach of Theorem \ref{onelevelunitary} relies on the following explicit
formula, which essentially converts a sum over zeros of an
$L$-function to a sum over primes. As it is similarly to that of \cite[Lemma 2.3]{G&Zhao2}, we omit its proof here.
\begin{lemma}
\label{lem2.4}
   Let $\phi(x)$ be an admissible function whose
Fourier transform $\hat{\phi}(u)$ has compact support in $[-2, 2]$. Let $\Lambda_K$ be the von Mangoldt function in $K$. Then for $q \in \mathcal{O}_K, (q,2)=1$ and any primitive character $\chi$ modulo $q$ satisfying $\chi(-1)=-1$, we have
\begin{equation*}
S(\widetilde{\chi}, \phi) =\int\limits^{\infty}_{-\infty}  \phi(t) \dif t-\frac 1{\log X}\sum_{(n)}\frac {\Lambda_K(n)}{\sqrt{N(n)}}\hat{\phi}\left( \frac {\log N(n)}{\log N(q)} \right)\left ( \widetilde{\chi}(n)  +\overline{\widetilde{\chi}}(n)\right )+O\left(\frac{1}{\log
N(q)}\right).
\end{equation*}
\end{lemma}

\section{Proof of Theorem ~\ref{moments}}

\subsection{Evaluation of the first moment}
\label{sec3.1}

   Since any integral non-zero ideal $\mathcal{A}$ co-prime to $2$ in $\mathcal{O}_K$ has a unique primary generator $a$,
we apply the approximate functional equation \eqref{approxfuneq} and the orthogonality relations Lemma \ref{ortho} to get that
\begin{align*}
  \sumstar_{\substack{\chiq \\ \chi(-1)=-1}} L \left( \frac{1}{2}, \widetilde{\chi} \right) =& \sum_{\substack {n \equiv 1 \bmod (1+i)^3}}
\frac{1}{\sqrt{N(n)}}V \left( \frac {N(n)}{x} \right)\sumstar_{\substack{\chiq \\ \chi(-1)=-1}}\widetilde{\chi}(n)  \\
& \hspace*{2cm} +\frac 1{(8N(q))^{1/2}}\sum_{\substack {n \equiv 1 \bmod (1+i)^3}}
\frac{1}{\sqrt{N(n)}}V \left( \frac {N(n)x}{32N(q)} \right) \sumstar_{\substack{\chiq \\ \chi(-1)=-1}}\overline{\widetilde{\chi}}(n)g(\widetilde{\chi}) \\
=& S_{1,1}+ S_{1,2}+S_{1,3}+ S_{1,4},
\end{align*}
   where
\begin{align*}
  S_{1,1} &= \frac 12\sum_{\substack{d|q \\ d \equiv 1 \bmod (1+i)^3}} \mu_{[i]}(q/d) \varphi(d) \sum_{\substack{n\equiv 1
\bmod (1+i)^3d\\(n,q)=1}}\frac{1}{\sqrt{N(n)}}V \left( \frac {N(n)}{x} \right), \\
 S_{1,2} &=-\frac 12\sum_{\substack{d|q \\ d \equiv 1 \bmod (1+i)^3}} \mu_{[i]}(q/d) \varphi(d) \sum_{\substack{n\equiv -1
\bmod d\\ n\equiv 1
\bmod (1+i)^3 \\ (n,q)=1}}\frac{1}{\sqrt{N(n)}}V \left( \frac {N(n)}{x} \right), \nonumber \\
 S_{1,3} &= \frac 12 \cdot \frac 1{(8N(q))^{1/2}}\sum_{\substack{d|q \\ d \equiv 1 \bmod (1+i)^3}} \mu_{[i]}(q/d) \varphi(d) \sum_{n\equiv 1
\bmod (1+i)^3}
\frac{1}{\sqrt{N(n)}} \\
& \hspace*{2cm} \times V \left( \frac {N(n)x}{32N(q)} \right)\sum_{\substack {x \bmod (1+i)^3q \\ x \equiv n \bmod d}}\widetilde{e}\leg{x}{(1+i)^3q}, \nonumber\\
 S_{1,4} &=-\frac 12 \cdot \frac 1{(8N(q))^{1/2}}\sum_{\substack{d|q \\ d \equiv 1 \bmod (1+i)^3}} \mu_{[i]}(q/d) \varphi(d) \sum_{n\equiv 1
\bmod (1+i)^3}
\frac{1}{\sqrt{N(n)}} \\
& \hspace*{2cm} \times V \left( \frac {N(n)x}{32N(q)} \right)\sum_{\substack {x \bmod (1+i)^3q \\ x \equiv -n \bmod d}}\widetilde{e}\leg{x}{(1+i)^3q}. \nonumber
\end{align*}

   As $\widetilde{e}(c) \ll 1 $ for $c \in \mathcal{O}_K$, we have that
\begin{align*}
  \sum_{\substack {x \bmod (1+i)^3q \\ x \equiv n \bmod d}}\widetilde{e}\leg{x}{(1+i)^3q} \ll \sum_{\substack {x \bmod (1+i)^3q \\ x \equiv n \bmod d}}1
\ll \frac {N(q)}{N(d)}.
\end{align*}

   It follows that
\begin{align*}
  S_{1,3} \ll  N(q)^{1/2}\sum_{\substack{d|q \\ d \equiv 1 \bmod (1+i)^3}} \mu^2_{[i]}(q/d) \frac {\varphi(d)}{N(d)}  \sum_{n\equiv 1
\bmod (1+i)^3}
\frac{1}{\sqrt{N(n)}}V \left( \frac {N(n)x}{32N(q)} \right) \ll \frac {N(q)^{1+\varepsilon}}{x^{1/2}}2^{\omega(q)}.
\end{align*}

   Similarly, we also have
\begin{align*}
  S_{1,4} \ll  \frac {N(q)^{1+\varepsilon}}{x^{1/2}}2^{\omega(q)}.
\end{align*}

  In the evaluation of $S_{1,1}$, we write $n=td+1$  with $t \in \mathcal{O}_K$. The term $t=0$ gives the main term:
\begin{align*}
  M_1 &= \frac 12\sum_{\substack{d|q \\ d \equiv 1 \bmod (1+i)^3}} \mu_{[i]}(q/d) \varphi(d) V \left( \frac {1}{x} \right) =\frac 12\sum_{\substack{d|q \\ d \equiv 1 \bmod (1+i)^3}} \mu_{[i]}(q/d) \varphi(d) \left( 1+O \left( x^{-1/2+\varepsilon} \right) \right) \\
 &=\frac 12\psi^*(q)+O\left( N(q)x^{-1/2+\varepsilon} \right),  \nonumber
\end{align*}
  where we have used Corollary \ref{psi} and the fact that
\begin{align}
\label{phi}
   \sum_{\substack{d|q \\ d \equiv 1 \bmod (1+i)^3}} \varphi(d) =N(q).
\end{align}

 To treat the contribution from the terms $n \neq 1$ in $S_{1,1}$, we need the following lemma.
\begin{lemma}
\label{arithmeticprog}
 Let $m, n \in \mz[i]$ satisfying $N(m+n) \geq N(n)$, then we have

\begin{align}
\label{normineq}
  N(m+n) \geq \frac {N(m)}{64}.
\end{align}
\end{lemma}
\begin{proof}
   The assertion of the Lemma is clearly true when $N(n) \geq \frac {N(m)}{64}$. We may therefore assume that $N(n) \leq \frac {N(m)}{64}$. Writing $m=a+bi, n=c+di$ with $a,b, c,d \in \mz$, we see that $N(n) \leq \frac {N(m)}{64}$ is equivalent to
\begin{align*}
 \frac {a^2+b^2}{64} \geq c^2+d^2.
\end{align*}
   We deduce from this that
\begin{align}
\label{maxcd}
 \max \{ |c|, |d| \} \leq \frac {\sqrt{a^2+b^2}}{8}.
\end{align}
   Writing \eqref{normineq} in terms of $a,b,c,d$, we find that it suffices to show
\begin{align}
\label{3.6}
  a^2+2ac+b^2+2bd \geq \frac {a^2+b^2}{64}.
\end{align}
   Applying \eqref{maxcd}, we see that
\begin{align*}
  a^2+2ac+b^2+2bd \geq  a^2+b^2-(|a|+|b|)\frac {\sqrt{a^2+b^2}}{4}.
\end{align*}
   As the above inequality implies inequality \eqref{3.6}, the assertion of the lemma now follows.
\end{proof}

  Applying Lemma \ref{normineq} to the case $n=td+1$ with $t \neq 0$, we see that in this case $N(td+1) \geq N(td)/64$. In view of the rapid decay of $V$ in \eqref{2.07}, we may further assume that $N(n) \leq x^{1+\varepsilon}$ for any $\varepsilon>0$. This implies that $N(td) \leq 64x^{1+\varepsilon}$. We then deduce that the terms with $t \neq 0$ in $S_{1,1}$ contribute an amount that is
\begin{align*}
    \ll \sum_{\substack{d|q \\ d \equiv 1 \bmod (1+i)^3}}  \mu^2_{[i]}(q/d) \varphi(d) \sum_{\substack{0 \neq N(td) \leq 64x^{1+\varepsilon}}}\frac{1}{\sqrt{N(td)}} \ll 2^{\omega(q)}x^{1/2+\varepsilon}.
\end{align*}

Thus, we have
\begin{align*}
  S_{1,1} = \frac 12 \psi^*(q)+O\left( N(q)x^{-1/2+\varepsilon}+2^{\omega(q)}x^{1/2+\varepsilon} \right).
\end{align*}

  Now, to estimate $S_{1,2}$, we write $n=td-1$  with $t \in \mathcal{O}_K$. Note that in this case $t \neq 0$ since $-1$ is not primitive.
The treatment of the contribution from these $t \neq 0$ terms is similar to that of $S_{1,1}$ and we arrive at
\begin{align*}
    S_{1,2} \ll 2^{\omega(q)}x^{1/2+\varepsilon}.
\end{align*}

   We then conclude that
\begin{align*}
  \sumstar_{\substack{\chiq \\ \chi(-1)=-1}} L\left( \frac{1}{2} , \widetilde{\chi} \right) =\frac 12 \psi^*(q)+O\left( N(q)x^{-1/2+\varepsilon}+2^{\omega(q)}x^{1/2+\varepsilon}+\frac {N(q)^{1+\varepsilon}}{x^{1/2}}2^{\omega(q)} \right).
\end{align*}

   By setting $x=N(q)$, we obtain \eqref{firstmoment}.

\subsection{The main term of the second moment}

   To establish \eqref{secondmoment}, we note that it is shown in Section \ref{AFE} that $|L(1/2, \widetilde{\chi})|^2=2A(\widetilde{\chi})$ with $A(\widetilde{\chi})$ given in \eqref{achi}. Again writing any integral non-zero ideal $\mathcal{A}$ co-prime to $2$ in $\mathcal{O}_K$ in term of its unique primary generator $a$ and applying Lemma \ref{ortho}, we have
\begin{align*}
& \sumstar_{\substack{\chiq \\ \chi(-1)=-1}} \left| L(1/2, \widetilde{\chi} ) \right|^2 = 2 \sum_{\substack{ n,m \\ n, m \text{ primary}}}
\frac{1}{\sqrt{N(n)N(m)}}W \left( \frac {N(nm)}{N(q)} \right) \sumstar_{\substack{\chiq \\ \chi(-1)=-1}}\widetilde{\chi}(n)\overline{\widetilde{\chi}}(m) =S_{2,1}-S_{2,2},
\end{align*}
  where
\begin{align*}
S_{2,1} &= \sum_{\substack {d|q \\ d \text{ primary}}} \mu_{[i]}(d) \varphi(q/d) \sum_{\substack{n\equiv m
\bmod q/d\\n, m \text{ primary} \\(mn,q)=1}}\frac{1}{\sqrt{N(n)N(m)}}W \left( \frac {N(nm)}{N(q)} \right), \\
S_{2,2} &=\sum_{\substack {d|q \\ d \text{ primary}}} \mu_{[i]}(d) \varphi(q/d) \sum_{\substack{n\equiv -m
\bmod q/d\\n, m \text{ primary} \\(mn,q)=1}}\frac{1}{\sqrt{N(n)N(m)}}W \left( \frac {N(nm)}{N(q)} \right). \nonumber
\end{align*}

   We consider the terms $n=m$ in $S_{2,1}$.  These terms contribute
\begin{align*}
 M_2 = \sum_{\substack {d|q \\ d \text{ primary}}} \mu_{[i]}(d) \varphi(q/d) \sum_{\substack{n \text{ primary}\\(n,q)=1}}\frac{1}{N(n)}W \left( \frac {N(n)^2}{N(q)} \right).
\end{align*}

   We then apply Mellin inversion to get
\begin{align}
\label{M2int}
\begin{split}
  & \sum_{\substack{n \text{ primary}\\(n,q)=1}}\frac{1}{N(n)}W \left( \frac {N(n)^2}{N(q)} \right)  = \frac 1{2\pi i } \int\limits_{(2)} \sum_{\substack{n \text{ primary}\\(n,q)=1}} \frac{1}{N(n)^{1+2s}} N(q)^s  \widehat{W}(s) \dif s \\
 &= \frac 1{2\pi i } \int\limits_{(2)} \zeta_{K}(1+2s) \left ( \prod_{\mathfrak{p} | 2q} \left( 1 - N(\mathfrak{p})^{-(1+2s)} \right) \right )  N(q)^s  \widehat{W}(s) \dif s.
\end{split}
\end{align}
Here and in what follows, we use $\zeta_K(s)$ to denote the Dedekind zeta function of $K$ and $\mathfrak{p}$ to denote prime ideals in $\mathcal{O}_K$.  Moreover, $ \widehat{W}(s)$ is the Mellin transform of $W(t)$, so that
\begin{align*}
      \widehat{W}(s) =\int\limits^{\infty}_0W(t)t^s\frac {\dif t}{t}.
\end{align*}

     Using \eqref{w} and integration by parts implies that for $\Re(s)>0$,
\begin{align}
\label{Wanlext}
      \widehat{W}(s) =\frac 1{s}I(s), \quad I(s)= \int\limits^{\infty}_0W'(t)t^s \dif t.
\end{align}

     Note that \eqref{w} further implies that $I(0)=1$ and integration by parts implies that $I(s)$ is clearly analytic for $\Re(s)>-1$ and satisfies
\begin{align*}
      I(s) \ll \frac{1}{|1+s|}.
\end{align*}

     It follows that \eqref{Wanlext} gives an analytic extension of $\widehat{W}(s)$ to $\Re(s) >-1$ with a
simple pole at $s = 0$ with residue $1$ such that when $\Re(s) >-1$,
\begin{equation}
\label{eq:h}
 \widehat{W}(s) \ll \frac{1}{|s||1+s|}.
\end{equation}

   We now shift the line of integration in \eqref{M2int} to $\Re(s)=-1/4+\varepsilon$ and we encounter a double pole at $s=0$. The residue is easily seen (by taking note that the residue of $\zeta_K(s)$ at $s = 1$ is $\pi/4$) to be
\begin{align}
\label{M2main}
   \frac {\pi}{16} \frac{\varphi(q)}{N(q)}  \log N(q)+\frac {\pi}{8} \frac{\varphi(q)}{N(q)}\sum_{\substack{\mathfrak{p} |2q}}  \frac {\log N(\mathfrak{p})}{N(\mathfrak{p})-1}  +\frac{\varphi(q)}{N(q)}C_0,
\end{align}
  where $C_0$ is an explicitly computable positive constant.

   Since $\displaystyle \sum_{\substack{\mathfrak{p} |q }} \frac {\log N(\mathfrak{p})}{N(\mathfrak{p})-1}$
is the largest when $q$ is of the form $\displaystyle \prod_{N(\varpi) \leq y}\varpi$ for primes $\varpi$, it follows from this and the prime ideal theorem \cite[Theorem 8.9]{MVa1} that
\begin{align}
\label{primebound}
 \sum_{\substack{\mathfrak{p} |q }} \frac {\log N(\mathfrak{p})}{N(\mathfrak{p})-1} \ll 1+\log \omega(q).
\end{align}

   To estimate the remaining integral at $\Re(s)=-1/4+\varepsilon$, we shall use the convexity bound that (see \cite[Exercise 3, p. 100]{iwakow}) for $\Re(s) =-1/4+\varepsilon$,
\begin{align*}
  \zeta_K(1+2s) \ll \left( 1+|s|^2 \right)^{1/4+\varepsilon}.
\end{align*}

   Applying this and \eqref{eq:h} gives that the integral on the line $\Re(s) = -1/4+\varepsilon$ is $\ll N(q)^{-1/4+\varepsilon}$.
From this and \eqref{M2main}, we get
\begin{align*}
 M_2 =&  \left ( \frac {\pi}{16} \frac{\varphi(q)}{N(q)}  \log N(q)+\frac {\pi}{8} \frac{\varphi(q)}{N(q)}\sum_{\substack{\mathfrak{p} |2q}}  \frac {\log N(\mathfrak{p})}{N(\mathfrak{p})-1}  +\frac{\varphi(q)}{N(q)}C_0 \right ) \sum_{\substack {d|q \\ d \text{ primary}}} \mu_{[i]}(d) \varphi(q/d)  \\ &\hspace*{2cm} +O\left(N(q)^{-1/4+\varepsilon} \sum_{d|q} \mu^2_{[i]}(d) \varphi(q/d) \right).
\end{align*}

  We then deduce using \eqref{phi} and \eqref{primebound} that
\begin{align}
\label{maintermsec}
 M_2 = \left ( \frac {\pi}{16} \frac{\varphi(q)}{N(q)}  \log N(q)+\frac {\pi}{8} \frac{\varphi(q)}{N(q)}\sum_{\substack{\mathfrak{p} |2q}}  \frac {\log N(\mathfrak{p})}{N(\mathfrak{p})-1}  +\frac{\varphi(q)}{N(q)}C_0 \right )\psi^{*}(q)+O\Big(N(q)^{3/4+\varepsilon} \Big ).
\end{align}

\subsection{The error term of the second moment}

   We first note that the terms $n=m$ in $S_{2,2}$ can occur if and only if $2n \equiv 0 \pmod {q/d}$. As $(q, 2)=1$, this occurs if and only if $q/d | n$. It follows readily from this that the terms $n=m$ in $S_{2,2}$ contribute
\begin{align}
\label{error1}
  \ll 2^{\omega(q)}\log N(q).
\end{align}

   To treat the contributions from the terms $n \neq m$ in $S_{2,1}$ and $S_{2,2}$, we note the following

\begin{lemma}\label{lemmaofStefanicki}
We have for any $\varepsilon>0$,
\begin{align}
\label{offdiag}
\sum_{\substack{n\neq m\\ n\equiv m \bmod \ell\\
(nm,q)=1}}
\frac{1}{\sqrt{N(n)N(m)}}W\left( \frac {N(nm)}{N(q)} \right) \ll
\frac{N(q)^{1/2+\epsilon}}{N(l)}.
\end{align}
\end{lemma}
\begin{proof}
   We may assume that $N(m) \geq N(n)$. In view of the rapid decey of $W$ shown in \eqref{w}, we may further assume that $N(nm) \leq N(q)^{1+\varepsilon}$ for any $\varepsilon>0$. We then have
\begin{align}
\label{3.7}
 \sum_{\substack{n\neq m\\ n\equiv m \bmod \ell\\
(nm,q)=1}}
\frac{1}{\sqrt{N(n)N(m)}}W\left( \frac {N(nm)}{N(q)} \right) \ll \sum_{\substack{N(n) \leq N(q)^{1+\varepsilon}}}\frac{1}{\sqrt{N(n)}}\sum_{\substack{ m \neq n\\ m \equiv n \bmod \ell\\
N(n) \leq N(m) \leq  N(q)^{1+\varepsilon}/N(n)}}
\frac{1}{\sqrt{N(m)}}.
\end{align}

   We write $m=n+kl$ with $k \in \mathcal{O}_K$ and we apply Lemma \ref{arithmeticprog} to see that $N(kl) \leq 64N(m) \leq 64N(q)^{1+\varepsilon}/N(n)$. Thus, we have
\begin{align*}
 \sum_{\substack{ m \neq n\\ m \equiv n \bmod \ell\\
N(n) \leq N(m) \leq  N(q)^{1+\varepsilon}/N(n)}}
\frac{1}{\sqrt{N(m)}} \ll \frac{1}{\sqrt{N(l)}}\sum_{\substack{ 0 \neq N(k) \leq  64N(q)^{1+\varepsilon}/N(n)}}
\frac{1}{\sqrt{N(kl)}} \ll \frac{1}{N(l)\sqrt{N(n)}}N(q)^{1/2+\varepsilon}.
\end{align*}

   Applying this in \eqref{3.7}, we readily deduce \eqref{offdiag} and this completes the proof of the lemma.
\end{proof}

   It follows from Lemma \ref{lemmaofStefanicki} that the terms $n \neq m$ contribute in $S_{2,1}, S_{2,2}$
\begin{align} \label{error2}
  \ll 2^{\omega(q)}N(q)^{1/2+\varepsilon}.
\end{align}

Using \eqref{omega} and combining \eqref{maintermsec}, \eqref{error1} and \eqref{error2}, the proof of \eqref{secondmoment} is complete.

\section{Proof of Theorem ~\ref{onelevelunitary}}

   Applying Lemma \ref{lem2.4}, we see that it suffices to show that for any $\hat{\phi}$ supported in $(-2 + \varepsilon, 2-\varepsilon)$ with any $0 < \varepsilon < 1$,
\begin{align}
\label{4.1}
  \lim_{N(q) \rightarrow \infty} \frac{\widetilde{S}(q,\hat{\phi})}{N(q) \log N(q)}=0,
\end{align}
   where
\begin{align*}
   \widetilde{S}(q,\hat{\phi})=\sumstar_{\substack{\chiq \\ \chi(-1)=-1}}\sum_{\substack{n \text{ primary}}}\frac {\Lambda_K(n)}{\sqrt{N(n)}}\hat{\phi}\left( \frac {\log N(n)}{\log N(q)} \right)\left ( \widetilde{\chi}(n)  +\overline{\widetilde{\chi}}(n)\right ).
\end{align*}

   Applying Lemma \ref{ortho}, we see that
\begin{align*}
   \widetilde{S}(q,\hat{\phi}) =& \sum_{\substack{d|q \\ d \equiv 1 \bmod (1+i)^3}}\mu_{[i]}(q/d) \varphi(d)\sum_{\substack{n \text{ primary} \\ n \equiv 1 \bmod d }}\frac {\Lambda_K(n)}{\sqrt{N(n)}}\hat{\phi}\left( \frac {\log N(n)}{\log q} \right) \\
& \hspace*{2cm}  -\sum_{\substack{d|q \\ d \equiv 1 \bmod (1+i)^3}} \mu_{[i]}(q/d) \varphi(d)\sum_{\substack{n \text{ primary} \\  n \equiv -1 \bmod d }}\frac {\Lambda_K(n)}{\sqrt{N(n)}}\hat{\phi}\left( \frac {\log N(n)}{\log q} \right).
\end{align*}

   Similar to the treatment of the case $n \neq 1$ in $S_{1,1}$ in Section \ref{sec3.1}, we have
\begin{align*}
   \sum_{\substack{n \text{ primary} \\  n \equiv \pm 1 \bmod d }}\frac {\Lambda_K(n)}{\sqrt{N(n)}}\hat{\phi}\left( \frac {\log N(n)}{\log q} \right) & \ll \sum_{\substack{n \text{ primary} \\ 1 < N(n) \leq q^{2-\epsilon}\\ n \equiv \pm 1 \bmod d  }}\frac {\log N(q)}{\sqrt{N(n)}} \ll
\frac {N(q)^{1-\varepsilon/2}\log N(q)}{N(d)}.
\end{align*}

   It follows that
\begin{align*}
   \widetilde{S}(q,\hat{\phi}) \ll  \sum_{\substack{d|q \\ d \equiv 1 \bmod (1+i)^3}} \mu^2_{[i]}(q/d) \varphi(d)\frac {N(q)^{1-\varepsilon/2}\log N(q)}{N(d)} \ll 2^{\omega(q)}N(q)^{1-\varepsilon/2}\log N(q).
\end{align*}

   In view of \eqref{omega}, the desired limit in \eqref{4.1} follows from the above estimation and this completes the proof of Theorem \ref{onelevelunitary}. \newline

\noindent{\bf Acknowledgments.} P. G. is supported in part by NSFC grant 11371043 and 11871082 and L. Z. by the FRG grant PS43707 and the Faculty Silverstar Award PS49334.

\bibliography{biblio}
\bibliographystyle{amsxport}

\vspace*{.5cm}

\noindent\begin{tabular}{p{8cm}p{8cm}}
School of Mathematical Sciences & School of Mathematics and Statistics \\
Beihang University & University of New South Wales \\
Beijing 100191 China & Sydney NSW 2052 Australia \\
Email: {\tt penggao@buaa.edu.cn} & Email: {\tt l.zhao@unsw.edu.au} \\
\end{tabular}

\end{document}